\documentclass [twoside,reqno,11pt]{amsart}

 \usepackage[a4paper, total={6in, 8in}]{geometry}
\usepackage[usenames]{color}
\usepackage[abbrev,nobysame,alphabetic]{amsrefs}
\usepackage{todonotes}
\usepackage{mathtools}  
\mathtoolsset{showonlyrefs}

\usepackage{amsmath,pdfsync,verbatim,graphicx,epstopdf,enumerate}
\usepackage[colorlinks=true]{hyperref}
\usepackage{cancel}
\usepackage{accents}
\hypersetup{linkcolor=blue,citecolor=red}
\usepackage{mathtools}  
\newcommand{\D}{\mathrm{d}}
\newcommand{\R}{\mathbb{R}}

\numberwithin{equation}{section}

\newcommand{\I}{\mathrm{i}}

\usepackage[english]{babel}
\usepackage{epsfig,graphics, graphicx}
\usepackage{subcaption, comment, bm}
\usepackage[percent]{overpic}
\usepackage{float}
\usepackage{color}
\usepackage[]{amsfonts}
\usepackage{amsopn}
\usepackage{amsmath,amsthm,amssymb}
\usepackage{relsize}
\usepackage[]{fancyhdr}
\usepackage[]{graphicx,wrapfig}
\usepackage{enumitem}
\usepackage{array}
\usepackage[T1]{fontenc}
\newtheorem{theorem}{Theorem}[section]

\newtheorem{lemma}[theorem]{Lemma}
\newtheorem{proposition}[theorem]{Proposition}

\theoremstyle{definition}

\newtheorem{remark}[theorem]{Remark}

\usepackage{amsmath,pdfsync,verbatim,graphicx,epstopdf,enumerate}

\newcommand{\FR}{\mathbb{R}}

\newcommand{\rt}{\mathbb{R}^2}
\newcommand{\Sb}{\mathbb{S}}
\newcommand{\F}{\mathcal{F}}
\newcommand{\E}{\mathcal{E}}
\title[Inversion formula, Unique continuation  property, and range characterization]{Inversion formula, Unique continuation  property, and range characterization of the mixed ray transform in $\R^2$}
\author[R. K. Mishra, S. K. Sahoo, and C. Thakkar]{Rohit Kumar Mishra, Suman Kumar Sahoo, and Chandni Thakkar}
\address{Department of Mathematics, IIT Gandhinagar, Gujarat, India}
\email{rohit.m@iitgn.ac.in, rohittifr2011@gmail.com}
\address{Department of Mathematics, Seminal for Applied Mathematics, ETH Z\"urich, Switzerland}
\email{suman.sahoo@math.ethz.ch,sumansahootifr@gmail.com}
\address{Department of Mathematics, IIT Gandhinagar, Gujarat, India}
\email{thakkar\_chandni@iitgn.ac.in}
\begin{document}
\begin{abstract}

In this article, we study various aspects of the mixed ray transform of $(k + \ell)$-tensor fields that are symmetric in its first $k$ and last $\ell$ indices. As a first result, we derive an inversion algorithm to recover the solenoidal part of the unknown tensor field using the normal operator of the mixed ray transform. Next, we establish a set of unique continuation results. In addition to these, we discuss the range characterization of the mixed ray transform as the final result.
\end{abstract}

\maketitle
\section{Introduction}

For an integer $m \geq 0$, let $\textit{T}^m = \textit{T}^m(\mathbb{R}^2)$ be the space of $m$-tensor fields defined on $\mathbb{R}^2$ and $\mathit{S}^m = \mathit{S}^m(\mathbb{R}^2)$ be its subspace consisting of symmetric $m$-tensor fields. Furthermore,  let  $C^\infty (\mathit{S}^m)$ denote the space of smooth symmetric $m$-tensor fields and $C_c^\infty (\mathit{S}^m)$ be the space of compactly supported tensor fields in $C^\infty (\mathit{S}^m)$. The notation $C^{\infty}(\mathit{S}^k \times \mathit{S}^\ell)$ is used for the space of smooth tensor fields that are symmetric with respect to first $k$ and last $\ell$ indices in $\FR^2$. Moreover, $T\Sb^1$ is the tangent bundle of the unit circle in $\rt$. The space of straight lines in $\mathbb{R}^2$ is  parameterized by 
\[T\mathbb{S}^{1}= \left\{(x,\xi) \in \mathbb{R}^2 \times \mathbb{S}^{1}: \left<x, \xi\right> = 0\right\}.\]
Here $\langle \cdot, \cdot \rangle$ denotes the usual dot product in $\mathbb{R}^2$. Any point $(x,\xi) \in T\Sb^1$ determines a unique line $\displaystyle\{ x+t\xi: t \in \mathbb{R}\}$.

Let $f \in C_c^{\infty}( S^k\times S^\ell)$, the \emph{mixed ray transform} (MiRT) $L_{k,\ell}:C_c^{\infty}( S^k\times S^\ell) \rightarrow C^{\infty} (T\Sb^1) $ is a bounded linear operator which is defined as:
\begin{equation}\label{eq:mixed ray transform}
    \begin{aligned}
\mathit{L}_{k, \ell}f (x, \xi) &= \int_{-\infty}^\infty f_{i_1 \dots i_k j_1 \dots j_\ell} (x + t \xi)\xi^{i_1} \cdots \xi^{i_k} \eta^{j_1} \cdots \eta^{j_\ell} \,dt,
\end{aligned}
\end{equation}
where $\eta $ is a unit vector orthogonal to $\xi = (\xi_1,\xi_2)$. Since we are working in two dimensions, we can take  $\eta = \xi^\perp = (-\xi_2, \xi_1)$. Note that we use Einstein summation convention in equation \eqref{eq:mixed ray transform}, and we will use it throughout the article. If we take $\ell = 0$  and $f\in C_c^{\infty}(S^m)$, then the equation \eqref{eq:mixed ray transform} gives the classical ray transform of symmetric $m$ tensor fields, which has a wide range of applications in medical imaging; see \cites{Sharafutdinov_1994,PSU_book}, denoted as $I_m$ and defined by
 \begin{align}\label{def_of_Im}
     I_mf(x,\xi)= \int_{-\infty}^\infty f_{i_1\cdots i_m}(x+t\xi)\, \xi^{i_1}\cdots \xi^{i_m}\, dt. 
 \end{align}
 The mixed ray transform, as originally introduced by Sharafutdinov in \cite{Sharafutdinov_1994}*{Chapter 7} for the analysis of inverse problems associated with shear waves in quasi-isotropic elastic media, serves as a fundamental tool. Further elucidation on the acquisition of mixed ray transform data from real-world observations can be found in \cite{Uniqueness_stability_MRT}. The aim of this article is to study the mixed ray transform (MiRT) and establish its invertibility, unique continuation properties (UCP), and the range characterization. More specifically, we establish the following results, see Section \ref{sec: definitions} for precise statements and more details:
\begin{enumerate}
    \item Inversion formula using the normal operator of the MiRT.
    \item Unique continuation property of the MiRT.
    \item  Range characterization of the MiRT.
\end{enumerate}
The inversion formula for the mixed ray transform of $f$ using the normal operator is inspired by Sharafutdinov's inversion formula for the longitudinal ray transform (LRT), derived in \cite{Sharafutdinov_1994}*{Section 2.12}. A different kind of inversion technique for MiRT is considered in \cite{Inversion_MRT}, where authors recover the unknown symmetric $m$-tensor field fully from the knowledge of three transforms: the longitudinal ray transform, transverse ray transform, and mixed ray transform in $\mathbb{R}^2$.  In \cite{Uniqueness_stability_MRT}, the authors proved a uniqueness and stability result using the normal operator of the mixed ray transform in dimensions $n\ge 3$ when $ f\in S^1\times S^1 \cup S^2\times S^2$ on a compact Riemannian manifold. Additionally, we refer to the works \cite{Joonas_2023,Rohit_Francois_Venky,Francois_2016,PSU,Salo_Uhlmann_JDG} for more results on ray transforms in various settings and the references therein.

The study of  unique continuation property of the ray transform of functions, also known as the interior tomography problem, has roots dating back to at least \cite{matias_ucp_2008}. In \cite{matias_ucp_2008}, the authors focused on reconstructing a compactly supported attenuation in a region based on its ray transform, assuming prior knowledge of the attenuation in a smaller region in $\rt$.  We refer to the works \cites{matias_semilocal_ucp,Katsevich_stability,Railo_ucp_radon} for related UCP results and stability estimates. UCP for the ray transform of compactly supported distributions and covector-valued distributions are proved in \cite{UCP_functions} and \cite{UCP_vector_fields}, respectively.  These findings are subsequently generalized for the longitudinal ray transform, momentum ray transform, and transverse ray transform of symmetric $m$-tensor fields in \cite{UCP_2022}. 

The approach in both \cite{UCP_vector_fields} and \cite{UCP_2022} involves reducing the unique continuation result of the ray transform of symmetric higher-ordered tensor fields to that of scalar functions and then applying results available for scalar functions in \cite{UCP_functions}, though with differing proof techniques. Our goal is to obtain unique continuation results for the mixed ray transform. We achieve this by leveraging the correlation established in \cite{kernel_MRT} between the mixed ray transform of $f$ and the longitudinal transform of a tensor field $g$, where $g$ is contingent on $f$. Moreover, in \cite{joonas2023unique} the authors introduced the fractional momentum ray transform acting on functions and studied the associated unique continuation property. 

The study of range characterization of the ray transform or the Radon transform is a classical topic and goes back to \cite{Helgason_1999,Gelfand_v_5}. We refer to the work \cite{Pantyukhina_range_2d} for the range of ray transform acting on symmetric tensor fields. We use similar ideas to obtain the range of the MiRT acting on tensor fields belonging to $S^k \times S^{\ell}.$  We also refer to the works \cite{Sharafutdinov_range_sobolev_space,krishnan2024ray, Momentum_ray_transform_2_2020, Rohit_2021, Rohit_Francois_JST, Assylbekov_Monard_Uhlmann, Sadiq_2023, Pestov_and_Uhlmann, PSU_IMRN, AMU_JMPA} for the range characterization of the ray transform, and  related  integral transforms.

The article is organized as follows. In Section \ref{sec: definitions}, we recall some notations, definitions, and state the main results of this article.  Section \ref{sec:proof_of_main_results} contains the proof of main results, which is further divided into three subsections. In subsection \ref{sec: inversion formula}, we present the inversion formula, then in subsection \ref{sec: UCP}, we prove the unique continuation property, and finally subsection \ref{sec:range_characterization} provides a proof for the range characterization result.

\section{Statement of main results}\label{sec: definitions}

 We first recall some notations and definitions which will be repeatedly used throughout the article. A detailed discussion of these definitions and operators can be found in Sharafutdinov's book, \cite{Sharafutdinov_1994}. 
 
The \textit{tensor product} of $u \in \textit{T}^m$ and $v \in \textit{T}^n$ is a tensor $u \otimes v \in \textit{T}^{m + n}$ and is defined by $$u \otimes v (x_1, \dots, x_m, x_{m + 1}, \dots, x_{m + n}) = u (x_1, \dots, x_m) v (x_{m + 1}, \dots, x_{m + n}).$$

 \noindent  The \textit{symmetrized tensor product} $u \odot v \in \textit{S}^{m + n}$ of $u \in \textit{S}^m$ and $v \in \textit{S}^n$ is given by
\begin{align*}
  u \odot v (x_1, \dots, x_m, x_{m + 1}, \dots, x_{m + n}) = \sigma (1, \dots, m + n) u (x_1, \dots, x_m) v (x_{m + 1}, \dots, x_{m + n})  
\end{align*}

\noindent where $\sigma (1, \dots, k)$ denotes the \textit{symmetrization operator}, an element of  the permutation group $\Pi_k$ on the set $\left\{1, \dots, k\right\}$, and defined as follows:
\begin{align}\label{def_of_sigma}
    \sigma (1, \dots, k) u (x_1, \dots, x_k) = \frac{1}{k!} \sum_{\pi \in {\Pi_k}} u (x_{\pi(1)}, \dots, x_{\pi(k)}).
\end{align}
\noindent For $p < k$, the operator of \textit{partial symmetrization }$\sigma (1, \dots, p)$ is defined in the following way:
$$\sigma (1, \dots, p) u (x_1, \dots, x_k) = \frac{1}{p!} \sum_{\pi \in {\Pi_p}} u (x_{\pi(1)}, \dots, x_{\pi(p)}, x_{p + 1}, \dots, x_k).$$

To state our main results we now recall some differential operators from \cite{Sharafutdinov_1994}*{Chapter 2}.   Define the operator of \textit{inner differentiation} (also known as symmetrized derivative) $\D : C^\infty (\mathit{S}^{m}) \rightarrow C^\infty (\mathit{S}^{m + 1})$ as follows:
\begin{equation} \label{def: inner differntiation}
\left(\D u\right) _{i_1 \dots i_{m + 1}} = \sigma (i_1, \dots, i_{m + 1}) \frac{\partial  u_{i_1 \dots i_m}}{\partial x^{i_{m + 1}}}.
\end{equation}
The \textit{Saint-Venant operator} $ R : C^{\infty} (\mathit{S}^{m}) \rightarrow C^{\infty} (\mathit{T}^{2m})$ is defined as follows:
\begin{equation}
    \left(Rf\right)_{i_1j_1 \dots i_mj_m} = \alpha(i_1j_1) \dots \alpha(i_mj_m)\frac{\partial^m f_{i_1 \dots i_m}}{\partial x_{j_1} \dots \partial x_{j_m}},
\end{equation}
where $\alpha(i_pi_q)$ known as the \textit{alternation operator} is given by 
$$\left(\alpha(i_pi_q) f\right)_{i_1 \dots i_r} = \frac{1}{2} \left( f_{i_1 \dots i_p \dots i_q \dots i_r} - f_{i_1 \dots i_q \dots i_p \dots i_r}\right).$$

\noindent Further, we introduce two additional operators crucial for understanding the kernel of the mixed ray transform, as discussed in \cite{kernel_MRT}. The first operator, denoted by $\lambda: C^\infty (\mathit{S}^{k - 1} \times \mathit{S}^{\ell - 1}) \rightarrow C^\infty (\mathit{S}^k \times \mathit{S}^\ell)$, is defined as follows:
\begin{equation} \label{def: lambda operator}
{(\lambda w)}{i_1 \dots i_k j_1 \dots j_\ell} = \sigma (i_1, \dots, i_k) \sigma (j_1, \dots, j_\ell) \delta_{i_1 j_1} w_{i_2 \dots i_k j_2 \dots j_\ell}.
\end{equation}
The second operator, denoted by $\D^\prime : C^\infty (\mathit{S}^{k-1} \times \mathit{S}^{\ell }) \rightarrow C^\infty (\mathit{S}^k \times \mathit{S}^\ell)$, is defined as follows:
\begin{equation}\label{def: d' operator}
{(\D^\prime u)}_{i_1 \dots i_k j_1 \dots j_\ell} = \sigma (i_1, \dots, i_k) \frac{\partial}{\partial x^{i_1}}u_{i_2 \dots i_k j_1 \dots j_{\ell}}.
\end{equation}
The dual of the operator $\lambda$ is denoted by $\mu : C^\infty (\mathit{S}^k \times \mathit{S}^\ell) \rightarrow C^\infty (\mathit{S}^{k - 1} \times \mathit{S}^{\ell - 1})$ and is given by:

\[{(\mu w)}_{i_1 \dots i_{k - 1} j_1 \dots j_{\ell - 1}} = w_{i_1 \dots i_k j_1 \dots j_\ell}\delta^{i_k j_\ell}.\]
And the dual of the differential operator $\D'$ is an operator $\delta' : C^\infty (\mathit{S}^{k} \times \mathit{S}^\ell) \rightarrow C^\infty (\mathit{S}^{k-1} \times \mathit{S}^{\ell })$ and defined as follows:

\[{(\delta' u)}_{i_2 \dots i_k j_1 \dots j_{\ell}} = \frac{\partial}{\partial x^{i_1}} u_{i_1 \dots i_k j_1 \dots j_\ell}.\]
With these notations established, we proceed to state the main results of this paper, beginning with the inversion formula utilizing the normal operator of the mixed ray transform.

\begin{theorem}\label{th:inversion_normal_ope}
    Let $f \in C_c^\infty (\mathit{S}^k \times \mathit{S}^\ell)$. 
    
Then, the components of $f^s$ can be written using the normal operator of the mixed ray transform as follows:
\begin{equation} \label{eq: inversion formula MRT}
f^s_{m_1 \dots m_k n_1 \dots n_\ell} = (-1)^{M_2}\, (-\Delta)^{1/2}(-\Delta')^{-N_2} P(D) \,  (N_{k,\ell}f)_{1\cdots 11\cdots 1}, 
\end{equation}
where $f^s$ is the divergence free and trace free component of $f$ according to Proposition \ref{thm: tensor field decomposition}. Here $ \Delta'= \partial^2_{x_2}$ is the Laplace operator in $x_2$ variable, and $P(D) = D^{M_2}_{x_1} \, D^{N_2}_{x_2}$, where $ D= \frac{1}{\I }\partial$, is a constant coefficient differential operator whose symbol is $y^{M_2}_1\, y^{N_2}_2$. Note that, here $M_2= M_2(m_1,\cdots, m_k)$ and $N_2= N_2(n_1, \cdots, n_{\ell})$; see equation \eqref{def_M_2_N_2} for more precise definition of $M_2$ and $N_2$. 
\end{theorem}
In the upcoming theorems, $\mathcal{D}' (\mathit{S}^m)$ and $\mathcal{E} ' (\mathit{S}^m)$ denote the space of symmetric $m$-tensor fields whose components are distributions and compactly supported distributions, respectively.
 \begin{theorem}\label{th:ucp_MiRT}
Let $U \subset \FR^2$ be a non-empty open set. Suppose $f \in \mathcal{E}' (\mathit{S}^k \times \mathit{S}^\ell)$ and  $f|_U = \lambda w' $, for some $w'\in \mathcal{E}' (\mathit{S}^{k-1} \times \mathit{S}^{\ell-1}) $ and the mixed ray transform $\mathit{L}_{k, \ell}f$  vanishes on all lines passing through $U$, then there exists $u \in \mathcal{E}' (\mathit{S}^{k-1} \times \mathit{S}^{\ell})$ and $w \in \mathcal{E}' (\mathit{S}^{k - 1} \times \mathit{S}^{\ell - 1})$ such that $f = \D^\prime u + \lambda w$.
 \end{theorem}
We say that a function $\phi$  vanishes to infinite order at a point $x_0 \in \R^2$ if $\phi$ is smooth in a neighborhood of  $x_0$, and the function  $\phi$ along with its partial derivatives of all orders vanishes at $x_0$, that is,  $ \partial^{\alpha} \phi (x_0)=0 $ for all multi-indices $ \alpha$.
\begin{theorem}\label{th_ucp_normal_op}
    Let $U \subset \FR^2$ be a non-empty open set. For some $f \in \mathcal{E}' (\mathit{S}^k \times \mathit{S}^\ell)$, let $f|_U = \lambda w'$, for some $ w'\in\E'(S^{k-1}\times S^{\ell-1}) $, and the normal operator $N_{k, \ell}f$  vanishes to infinite order at some point $x_0\in U$, then there exists $u \in \mathcal{E}' (\mathit{S}^{k-1} \times \mathit{S}^{\ell})$ and $w \in \mathcal{E}' (\mathit{S}^{k - 1} \times \mathit{S}^{\ell - 1})$ such that $f = \D^\prime u + \lambda w$.
\end{theorem}
The next result is a generalization of the above, indicating that if one knows the tensor field locally, along with its mixed ray transform, then the tensor field can be recovered globally.
\begin{theorem}\label{th_ucp_MiRT_1}
Let $U \subset \FR^2$ be a non-empty open set. For some $f \in \mathcal{E}' (\mathit{S}^k \times \mathit{S}^\ell)$, let $f|_U = \D' u+\lambda w'$ with $u \in \mathcal{E}' (\mathit{S}^{k-1} \times \mathit{S}^{\ell})$ and $w \in \mathcal{E}' (\mathit{S}^{k - 1} \times \mathit{S}^{\ell - 1})$, and the normal operator $N_{k, \ell}f$  vanishes to infinite order at some point $x_0\in U$,  then there exists $\Tilde{u} \in \mathcal{E}' (\mathit{S}^{k-1} \times \mathit{S}^{\ell})$ and $w \in \mathcal{E}' (\mathit{S}^{k - 1} \times \mathit{S}^{\ell - 1})$ such that $f = \D^\prime \Tilde{u} + \lambda w$.
\end{theorem}

We now state the range characterization of the mixed ray transform in two dimensions in the Schwartz space. Let $\zeta = \xi_1+\I \xi_2 \in \mathbb{C}, p\in \R$, and   $ \mathcal{S}(S^k\times S^{\ell}) \subset C^{\infty}(S^k\times S^{\ell}) $ be the Schwartz space of tensor fields of order $(k+\ell)$ which are symmetric in first $k$ and last $\ell$ indices. The Schwartz space $\mathcal{S}(T\Sb^1) \subset C^\infty(T\Sb^1)$ is also defined in the same way with the understanding that the decays and all the derivatives are with respect to the first component. 

\begin{theorem}\label{th:range_characterization}
Let $ \phi \in \mathcal{S}(T\Sb^1)$ and $ k,\ell$ be two non-negative integers. Then $\phi $ belongs to the range of the operator 
\begin{align}
     f \in \mathcal{S}(S^k\times S^{\ell}) \mapsto L_{k,\ell}f \in \mathcal{S}(T\Sb^1) 
\end{align}
if and only if the following two conditions are satisfied
\begin{itemize}
    \item [(i)] $\phi(\I p \zeta, -\zeta)= (-1)^{k+\ell} \phi(\I p \zeta,\zeta)  $.
    \item [(ii)] $ \int_{\R} p^r \phi(\I p \zeta,\zeta) dp =P_{k\ell r}(\zeta) $, \quad \mbox{for}\quad $ |\zeta|=1$,
\end{itemize}
where $P_{k\ell r}$ is a homogeneous polynomial of degree $ k+\ell+r$ in $ (\zeta, \Bar{\zeta}).$
\end{theorem}
\begin{remark}
The MiRT defined above can be extended to the space $\mathbb{R}^2 \times \mathbb{R}^2\setminus \{0\}$. Let us denote this extended operator by $\widetilde{\mathit{L}_{k, \ell}}f$. The operators $\widetilde{\mathit{L}_{k, \ell}}f$ and $\mathit{L}_{k, \ell}f$ are equivalent which is evident from the below expression:
\[\widetilde{\mathit{L}_{k, \ell}}f (x, \xi) = |\xi|^{k + \ell - 1} \mathit{L}_{k, \ell}f \left(x, \frac{\xi}{|\xi|}\right).\]
The above expression is similar to \cite{Sharafutdinov_1994}*{Equation 2.10.8}, and it shows that knowing one of the operators gives all the information about the other operator. Hence, we will not distinguish between $\widetilde{\mathit{L}_{k, \ell}}f$ and $\mathit{L}_{k, \ell}f$, and just use the notation $\mathit{L}_{k, \ell}f$ in the further discussion.
\end{remark}

\section{Proof of main results}\label{sec:proof_of_main_results}
The aim of this section is to prove the main results stated in the preceding section. To this end, we define the  linear operator $\textit{A}: C^\infty (\mathit{S}^k \times \mathit{S}^\ell) \rightarrow C^\infty (\mathit{S}^k \times \mathit{S}^\ell)$ as defined in \cite{kernel_MRT}:
\begin{align}\label{def_of_A}
    \textit{A}f_{i_1 \dots i_k j_1 \dots j_\ell} = {(- 1)}^{\ell - N(j_1, \dots, j_\ell)} f_{i_1 \dots i_k \delta(j_1, \dots, j_\ell)},
\end{align}
where $N(j_1, \dots, j_\ell)$ is the number of 1's in $(j_1, \dots, j_\ell)$ and $\delta$ maps 1's in $(j_1, \dots, j_\ell)$ to 2's and vice versa. So $\delta(j_1, \dots, j_\ell)$ is the tuple obtained from applying $\delta$ to $(j_1, \dots, j_\ell)$.

All the results presented in this article can be demonstrated in the context of distribution by employing routine density arguments. However, for the sake of clarity and better readability, we opt to confine our analysis to the smooth case setting only.

\subsection{Inversion formula}\label{sec: inversion formula}
The aim of this section is to first derive an explicit formula for the normal operator of MiRT over $(k+\ell)$-tensor fields and then use the obtained formula to proof of Theorem \ref{th:inversion_normal_ope}. To be more specific, we show that the normal operator of MiRT can be used to explicitly recover the componentwise Fourier transform of the solenoidal part of the unknown tensor field. That can be used to get the solenoidal part of the unknown tensor field by Fourier inversion. To achieve this, we follow the idea of Sharafutdinov \cite{Sharafutdinov_1994}*{Chapter 2} in the case of LRT.

\begin{lemma}\label{lem;normal_operator}
Let $ f\in C_c^{\infty}(S^k \times S^\ell)$, the normal operator $ N_{k,\ell}$ of the mixed ray transform is given by 
\begin{align}\label{Normal operator MRT}
({\textit{N}_{k,\ell} f})_{i_1 \dots i_k j_1 \dots j_\ell} (x) = 2 f_{m_1 \dots m_k n_1 \dots n_\ell} (x)  *  {(- 1)}^{(J_1 + N_1)} \quad \frac{x_1^{(I_1 + J_2 + M_1 + N_2)} x_2^{(I_2 + J_1 + M_2 + N_1)}}{|x|^{2k + 2\ell+ 1}},
\end{align}
where $I_i, J_i, M_i, N_i$ are defined in equation \eqref{def_M_2_N_2}.
\end{lemma}
\begin{proof}
We start with deriving the expression for the formal $\textit{L}^2-$ adjoint $({\mathit{L}_{k, \ell}})^*$ of $\mathit{L}_{k, \ell}$. For $f \in C_c^\infty (\mathit{S}^k \times \mathit{S}^\ell)$ and $g \in C^\infty (T\mathbb{S}^{1})$, consider
\begin{align*}
    \left<f, ({\mathit{L}_{k, \ell}})^* g\right>_{\FR^2} &= \left<\mathit{L}_{k, \ell} f, g \right>_{T\mathbb{S}^{1}}\\
    &= \int_{\mathbb{S}^{1}} \int_{\xi^\perp} \mathit{L}_{k, \ell} f (x, \xi) g(x, \xi) \,dx \,dS_{\xi}\\
    &= \int_{\mathbb{S}^{1}} \int_{\xi^\perp} \left\{\int_{-\infty}^\infty f_{i_1 \dots i_k j_1 \dots j_\ell} (x + t \xi) \xi^{i_1} \cdots \xi^{i_k} \eta^{j_1} \cdots \eta^{j_\ell} \,dt\right\} g(x, \xi) \,dx \,dS_{\xi}\\
    &= \int_{\mathbb{S}^{1}} \int_{\FR^2} f_{i_1 \dots i_k j_1 \dots j_\ell} (z) \xi^{i_1} \cdots \xi^{i_k} \eta^{j_1} \cdots \eta^{j_\ell} g(z - \left<z, \xi\right> \xi, \xi) \,dz \,dS_{\xi}.
\end{align*}
Here $dS_{\xi}$ denotes surface measure on $\mathbb{S}^1$ and we used the change of variables $x + t \xi = z$ for $x \in \xi^\perp$ and $t \in \FR$ for each fixed $\xi \in \mathbb{S}^{1}$. Further, interchanging the order of integration gives 
$$\left<f, ({\mathit{L}_{k, \ell}})^* g\right>_{\FR^2} = \int_{\FR^2} f_{i_1 \dots i_k j_1 \dots j_\ell} (z) \left\{\int_{\mathbb{S}^{1}} \xi^{i_1} \cdots \xi^{i_k} \eta^{j_1} \cdots \eta^{j_\ell} g(z - \left<z, \xi\right> \xi, \xi) \,dS_{\xi} \right\} \,dz.$$
 Hence, the $\textit{L}^2$ adjoint $({\mathit{L}_{k, \ell}})^*$ of $f$ is given by
\begin{equation}
    (({\mathit{L}_{k, \ell}})^* f)_{i_1 \dots i_k j_1 \dots j_\ell} (x) = \int_{\mathbb{S}^{1}} \xi^{i_1} \cdots \xi^{i_k} \eta^{j_1} \cdots \eta^{j_\ell} g(x - \left<x, \xi\right> \xi, \xi) \,dS_{\xi}.
\end{equation}
Using this adjoint operator, we can extend the definition of MiRT for tensor fields whose components are compactly supported distributions. 

Then the operators ($L_{k, \ell}$ and ${(L_{k, \ell})}^*$) defined above can be extended in the distribution setting by replacing $C^\infty (\mathit{S}^m)$ with $\mathcal{D}' (\mathit{S}^m)$ and $C_c^\infty (\mathit{S}^m)$ with $\mathcal{E} ' (\mathit{S}^m)$, that is, $\mathit{L}_{k, \ell}: \mathcal{E}' (\mathit{S}^k \times \mathit{S}^\ell) \rightarrow \mathcal{D}' (T\mathbb{S}^{1})$ is given by:
\[\left<\mathit{L}_{k, \ell} f, g \right> = \left< f, ({\mathit{L}_{k, \ell}})^* g\right>\]
for $f \in \mathcal{E}' (\mathit{S}^k \times \mathit{S}^\ell)$ and $g \in C_c^\infty (T\mathbb{S}^{1})$.
\noindent Further, let $\textit{N}_{k,\ell} = ({\mathit{L}_{k, \ell}})^* \mathit{L}_{k, \ell}: C_c^\infty (\mathit{S}^k \times \mathit{S}^\ell) \rightarrow C^\infty (\mathit{S}^k \times \mathit{S}^\ell)$ be the normal operator of the MiRT of a symmetric 2-dimensional $m$-tensor field, then
\begin{align*}
    ({\textit{N}_{k,\ell} f})_{i_1 \dots i_k j_1 \dots j_\ell} (x) &= ({\mathit{L}_{k, \ell}}^* \mathit{L}_{k, \ell} f)_{i_1 \dots i_k j_1 \dots j_\ell} (x)\\
    &= \int_{\mathbb{S}^{1}} \xi^{i_1} \cdots \xi^{i_k} \eta^{j_1} \cdots \eta^{j_\ell} \mathit{L}_{k, \ell} f(x - \left<x, \xi\right> \xi, \xi) \,dS_{\xi}.
\end{align*}
 Since $\mathit{L}_{k, \ell} f(x + t\xi, \xi) = \mathit{L}_{k, \ell} f(x, \xi)$, we get
\begin{align*}
    ({\textit{N}_{k,\ell} f})_{i_1 \dots i_k j_1 \dots j_\ell} (x) &= \int_{\mathbb{S}^{1}} \xi^{i_1} \cdots \xi^{i_k} \eta^{j_1} \cdots \eta^{j_\ell} \mathit{L}_{k, \ell} f(x, \xi) \,dS_{\xi}\\
    &= \int_{\mathbb{S}^{1}} \xi^{i_1} \cdots \xi^{i_k} \eta^{j_1} \cdots \eta^{j_\ell} \left\{ \int_{-\infty}^\infty f_{m_1 \dots m_k n_1 \dots n_\ell} (x + t \xi) \xi^{m_1} \cdots \xi^{m_k} \eta^{n_1} \cdots \eta^{n_\ell} \,dt \right\} \,dS_{\xi}\\
    &= 2 \int_{\mathbb{S}^{1}} \int_0^\infty f_{m_1 \dots m_k n_1 \dots n_\ell} (x + t \xi) \xi^{i_1} \cdots \xi^{i_k} \xi^{m_1} \cdots \xi^{m_k} \eta^{j_1} \cdots \eta^{j_\ell} \eta^{n_1} \cdots \eta^{n_\ell} \,dt \,dS_{\xi}.
\end{align*}
Inserting $\eta = \xi^\perp = (-\xi_2, \xi_1)$ in the above equation, we obtain
\begin{align*}
   & ({\textit{N}_{k,\ell} f})_{i_1 \dots i_k j_1 \dots j_\ell} (x) \\&\quad = 2 \int_{\mathbb{S}^{1}} \int_0^\infty {(- 1)}^{N(j_1, \dots, j_\ell)} {(- 1)}^{N(n_1, \dots, n_\ell)} f_{m_1 \dots m_k n_1 \dots n_\ell} (x + t \xi)\\
    &\qquad\quad {\xi_1}^{N(i_1, \dots, i_k, m_1, \dots, m_k)} {\xi_2}^{2k - N(i_1, \dots, i_k, m_1, \dots, m_k)}{\xi_1}^{2\ell - N(j_1, \dots, j_\ell, n_1, \dots, n_\ell)} {\xi_2}^{N(j_1, \dots, j_\ell, n_1, \dots, n_\ell)} \,dt \,dS_{\xi}.
\end{align*}
Performing the change of variables $x + t \xi = y$, we get $t = |x - y|$ and $\xi = -\frac{\left(x - y\right)}{|x - y|}$. For $x - y = \left({(x - y)}_1, {(x - y)}_2\right)$, we get
\begin{align*}
    ({\textit{N}_{k,\ell} f})_{i_1 \dots i_k j_1 \dots j_\ell} (x) &= 2 {(- 1)}^{N(j_1, \dots, j_\ell)+N(n_1,\cdots,n_{\ell})} \int_{\FR^2} \frac{f_{m_1 \dots m_k n_1 \dots n_\ell} (y)}{{|x - y|}^{2k + 2\ell+ 1}}\\& \left\{  {{(x - y)}_1}^{2l + N(i_1, \dots, i_k, m_1, \dots, m_k) - N(j_1, \dots, j_\ell, n_1, \dots, n_\ell)} \right.\\
    &\qquad \left. {{(x - y)}_2}^{2k + N(j_1, \dots, j_\ell, n_1, \dots, n_\ell) - N(i_1, \dots, i_k, m_1, \dots, m_k)} \right\} \,dy.
\end{align*}
Next, we define some shorthand notations to simplify the above computation.

\begin{equation}\label{def_M_2_N_2}
\begin{aligned}
I_1 &= I_1 (i_1, \dots, i_k) = N (i_1, \dots, i_k), &&I_2 &&= I_2 (i_1, \dots, i_k) = k - N (i_1, \dots, i_k)\\
J_1 &= J_1 (j_1, \dots, j_\ell) = N (j_1, \dots, j_\ell), &&J_2 &&= J_2 (j_1, \dots, j_\ell) = \ell - N (j_1, \dots, j_\ell)\\
M_1 &= M_1 (m_1, \dots, m_k) = N (m_1, \dots, m_k), &&M_2 &&= M_2 (m_1, \dots, m_k) = k - N (m_1, \dots, m_k)\\
N_1 &= N_1 (n_1, \dots, n_\ell) = N (n_1, \dots, n_\ell), &&N_2 &&= N_2 (n_1, \dots, n_\ell) = \ell - N (n_1, \dots, n_\ell).
\end{aligned}
\end{equation}

 Using these notations, we rewrite the normal operator of MiRT as:
\begin{equation} \label{eq: normal operator}
({\textit{N}_{k,\ell} f})_{i_1 \dots i_k j_1 \dots j_\ell} (x) = 2\,{(- 1)}^{(J_1 + N_1)}\, f_{m_1 \dots m_k n_1 \dots n_\ell} (x)  *  \frac{x_1^{(I_1 + J_2 + M_1 + N_2)} x_2^{(I_2 + J_1 + M_2 + N_1)}}{|x|^{2k + 2\ell+ 1}}.
\end{equation}
The equation for the normal operator of MiRT is similar to the one derived in \cite[Equation 2.11.3]{Sharafutdinov_1994} for the operator LRT. This equation makes sense even for $f \in \mathcal{E}' (\mathit{S}^k \times \mathit{S}^\ell)$, hence the normal operator $\textit{N}_{k,\ell}$ can be extended to the space of distributions $\textit{N}_{k,\ell} : \mathcal{E}' (\mathit{S}^k \times \mathit{S}^\ell) \rightarrow \mathcal{D}' (\mathit{S}^k \times \mathit{S}^\ell)$.   
\end{proof}

\begin{proposition} \label{thm: tensor field decomposition}
    For any tensor field $f \in C_c^\infty (\mathit{S}^k \times \mathit{S}^\ell)$, there exists $f^s \in C^\infty (\mathit{S}^k \times \mathit{S}^\ell),\  u \in C^\infty (\mathit{S}^{k-1} \times \mathit{S}^{\ell})$ and $w \in C^\infty (\mathit{S}^{k - 1} \times \mathit{S}^{\ell - 1})$ such that:
\begin{equation}\label{tensor field decomposition}
   f = f^s + \D^\prime u + \lambda w,
\end{equation}
where, $\delta'f^s = \mu f^s = 0, \ \mu u = 0$, and $f^s,u \rightarrow 0$ as $|x| \rightarrow \infty$.
\end{proposition}
\begin{proof}  Given any tensor field  $ f\in C_c^\infty (\mathit{S}^k \times \mathit{S}^{\ell})$, we can define the tensor field $g= \sigma (Af) \in C_c^{\infty}(S^{k+\ell})$. Since $C_c^{\infty}(S^{k+\ell}) \subset \mathcal{S}(S^{k+\ell}) $, we now apply the solenoidal-potential decomposition on $g$ from \cite[Theorem 2.6.2]{Sharafutdinov_1994}. This implies there are smooth tensor fields $g^s $ and $v$ such that
\begin{align}\label{decomposition}
    g= g^s+d v, \quad \delta g^s=0, \quad g^s, v\rightarrow 0 \, \, \mbox{as} \, \, |x| \rightarrow \infty.
\end{align}
Next we write $f$ as:
\begin{align}
    f = A^{-1} \left( Af - \sigma Af + \sigma Af\right).
\end{align}
Next we substitute the expression of $g = \sigma Af$ from \eqref{decomposition}  into above equation and conclude 
\begin{align}
    f =  A^{-1} \left( Af - \sigma Af + (g^s+\D v)\right).
\end{align}
Furthermore, using the fact that $ A^{-1}= (-1)^{\ell} A$ (see \cite[Equation 4.6]{kernel_MRT}) from above we obtain that 
\begin{align}\label{eq_1}
    f = f + (-1)^{\ell+1} A \sigma Af + (-1)^{\ell} A (g^s+\D v).
\end{align}
By  \cite{kernel_MRT}*{Section $4.1$} we have that
\begin{align*}
 f + (-1)^{\ell+1} A \sigma Af= \lambda w,  \quad A(\D v-\D'v)= \lambda w', \quad\mbox{and} \quad \D'A=A\D'. 
\end{align*}
The combination of this with \eqref{eq_1} implies
\begin{align}
    f = f^s+ \D'V + \lambda W, \quad \mbox{where} \quad f^s= (-1)^{\ell} A g^s, \; V= (-1)^\ell A v,\; W = w + (-1)^{\ell} w'. 
\end{align}
To conclude the argument, it is enough to show that $ f^s$ satisfies $ \delta' f^s=\mu f^s=0$. Using the fact that $ \delta' A = A \delta'$ we have that $ \delta' f^s=0.$ To show that $\mu f^s=0 $ by Lemma \ref{adjoint_of_asigma} we observe that 
\begin{align}
    \mbox{Ker}\mu = (\mbox{Im}\lambda)^{\perp}= \mbox{Ker} (A \sigma A)^{\perp} = \mbox{Im} (A\sigma A)^* = \mbox{Im} A\sigma A.
\end{align}
Similarly, we have that $ \mu V =0$. Moreover, $V(x) \rightarrow 0$ as $|x| \rightarrow \infty $.
This completes the proof.
\smallskip

We now present an alternate proof of Proposition \ref{thm: tensor field decomposition}. Let $\Tilde{C}^\infty (\mathit{S}^k \times \mathit{S}^{\ell - 1}) \subset C^\infty (\mathit{S}^k \times \mathit{S}^{\ell - 1})$ be the space containing all $u$ such that $u \rightarrow 0$ as $|x| \rightarrow \infty$. Also, observe that an analog of \cite[Theorem 2.6.2]{Sharafutdinov_1994} holds for the operators $\D'$ and $\delta'$. That is,
 \begin{equation} \label{eq: direct sum 1}
        C_c^\infty (\mathit{S}^k \times \mathit{S}^{\ell}) = \text{Ker} \delta' \oplus \text{Im} \D'
    \end{equation} 
    for $\D' : \Tilde{C}^\infty (\mathit{S}^k \times \mathit{S}^{\ell - 1}) \rightarrow C^\infty (\mathit{S}^k \times \mathit{S}^{\ell})$. This implies that Im$\D'$ is closed. Since $\lambda$ and $\mu$ are algebraic operators, the decomposition
\begin{equation} \label{eq: direct sum 2}
C_c^\infty (\mathit{S}^k \times \mathit{S}^{\ell}) = \text{Ker} \mu \oplus \text{Im} \lambda
    \end{equation}
    
    holds. This implies Im$\lambda$ is closed which further implies that $\text{Im}\D' + \text{Im}\lambda$ is closed in $C_c^\infty (\mathit{S}^k \times \mathit{S}^{\ell})$. Hence, the following decomposition holds:
\[C_c^\infty (\mathit{S}^k \times \mathit{S}^{\ell}) = \left(\text{Ker} \delta' \cap \text{Ker} \mu\right) \oplus \left(\text{Im} \D' + \text{Im} \lambda\right).\]
    Hence, for any $f \in C_c^\infty (\mathit{S}^k \times \mathit{S}^{\ell}),$ there exists $f^s \in \text{Ker} \delta' \cap \text{Ker} \mu, u' \in \Tilde{C}^\infty (\mathit{S}^k \times \mathit{S}^{\ell - 1})$ and $w' \in C^\infty (\mathit{S}^{k - 1} \times \mathit{S}^{\ell - 1})$ such that 
 $$f = f^s + \D' u' + \lambda w'.$$
    Further, note that the decomposition \eqref{eq: direct sum 2} also holds for the space $\Tilde{C}^\infty (\mathit{S}^k \times \mathit{S}^{\ell - 1})$ and hence we can write
$$u' = u + \lambda w''; \quad \mu u = 0.$$
    Finally, the commutativity of operators $\D'$ and $\lambda$ gives the decomposition
    $$f = f^s + \D' u + \lambda w$$
    with $w = w' + \D' w''$.    
\end{proof}
\begin{lemma}\label{adjoint_of_asigma}
    The operator $ A\sigma A $ is self adjoint $i.e.,$ $(A\sigma A)^* = A\sigma A.$
\end{lemma}
\begin{proof}
    For $f, g \in C_c^\infty(\mathit{S}^k \times \mathit{S}^\ell)$, we want to show $\left<A \sigma A f, h\right> = \left<h, A \sigma A g\right>$. 
\begin{align*}
    \left<A \sigma A f, h\right> &= (A \sigma A f)_{i_1 \dots i_k j_1 \dots j_\ell} (h)_{i_1 \dots i_k j_1 \dots j_\ell}\\
    &= \sum_{r = 0}^k \sum_{s = 0}^\ell \binom{k}{r} \binom{\ell}{s} (A \sigma A f)_{\underbrace{1 \dots 1}_{r} \underbrace{2 \dots 2}_{k - r} \underbrace{1 \dots 1}_{\ell - s} \underbrace{2 \dots 2}_{s}} \hspace{1mm} (h)_{\underbrace{1 \dots 1}_{r} \underbrace{2 \dots 2}_{k - r} \underbrace{1 \dots 1}_{\ell - s} \underbrace{2 \dots 2}_{s}}\\
    &= \sum_{r = 0}^k \sum_{s = 0}^\ell \binom{k}{r} \binom{\ell}{s} {(-1)}^s (\sigma A f)_{\underbrace{1 \dots 1}_{r} \underbrace{2 \dots 2}_{k - r} \underbrace{1 \dots 1}_{s} \underbrace{2 \dots 2}_{\ell - s}} \hspace{1mm} (h)_{\underbrace{1 \dots 1}_{r} \underbrace{2 \dots 2}_{k - r} \underbrace{1 \dots 1}_{\ell - s} \underbrace{2 \dots 2}_{s}}\\
    &= \sum_{p = 0}^{k + \ell} (\sigma A f)_{\underbrace{1 \dots 1}_{p} \underbrace{2 \dots 2}_{k + \ell - p}} \left\{\sum_{\substack{r + s = p \\ 0 \leq r \leq k \\ 0 \leq s \leq \ell}} \binom{k}{r} \binom{\ell}{s} {(- 1)}^{s} (h)_{\underbrace{1 \dots 1}_{r} \underbrace{2 \dots 2}_{k - r} \underbrace{1 \dots 1}_{\ell - s} \underbrace{2 \dots 2}_{s}}\right\}\\
    &= \sum_{p = 0}^{k + \ell} \binom{k + \ell}{p} {(-1)}^\ell (\sigma A f)_{\underbrace{1 \dots 1}_{p} \underbrace{2 \dots 2}_{k + \ell - p}} \hspace{1mm} (\sigma A h)_{\underbrace{1 \dots 1}_{p} \underbrace{2 \dots 2}_{k + \ell - p}}. 
\end{align*}
The last equality follows from the relation \eqref{eq: g and f relation with indices}. A similar calculation shows that
\begin{equation*}
    \left<h, A \sigma A g\right> = \sum_{p = 0}^{k + \ell} \binom{k + \ell}{p} {(-1)}^\ell (\sigma A f)_{\underbrace{1 \dots 1}_{p} \underbrace{2 \dots 2}_{k + \ell - p}} \hspace{1mm} (\sigma A h)_{\underbrace{1 \dots 1}_{p} \underbrace{2 \dots 2}_{k + \ell - p}}.
\end{equation*}
Hence, $A \sigma A$ is a self-adjoint operator. This completes the proof.
\end{proof}
\begin{remark}
For some $u \in C^\infty (\mathit{S}^{k-1} \times \mathit{S}^{\ell})$ and $w \in C^\infty (\mathit{S}^{k - 1} \times \mathit{S}^{\ell - 1})$, $\textit{N}_{k,\ell} (\D^\prime u) = 0$ and $\textit{N}_{k,\ell} (\lambda w) = 0$. Hence, we have $\textit{N}_{k,\ell} f = \textit{N}_{k,\ell} (f^s)$. 
\end{remark}

The proof of Theorem \ref{th:inversion_normal_ope} relies on the following lemma, which states (with some assumptions on $f^s$) that the Fourier transform of $f^s$ is completely determined explicitly if $\widehat{f^s}_{1\dots 1 1\dots 1}$ is known.

\begin{lemma} \label{lem: relation between components of f}
Let $h$ be a $(k+\ell)$-tensor satisfying $\mu h = 0$ and $y_m h_{m m_2 \dots m_k n_1 \dots n_\ell} = 0$ then $h$ is completely determined in terms of $h_{\underbrace{1 \dots 1}_k \underbrace{1 \dots 1}_\ell}$ as follows:
\[h_{m_1 \dots m_k n_1 \dots n_\ell} (y) = {(-1)}^{k - N(m_1,\dots,m_k)} \left(\frac{y_1}{y_2}\right)^{\left(k - \ell) - (N(m_1,\dots,m_k) - N(n_1,\dots,n_\ell)\right) } h_{\underbrace{1 \dots 1}_k \underbrace{1 \dots 1}_\ell} (y).\]
\end{lemma}

\begin{proof}
To simplify the calculation and to write things in compact form, we first express $h$ using the notations $M_1, M_2, N_1$, and $N_2$ defined earlier (in Lemma \ref{lem;normal_operator}). 
$$h_{m_1 \dots m_k n_1 \dots n_\ell} (y) = h_{\underbrace{1 \dots 1}_{M_1} \underbrace{2 \dots 2}_{M_2} \underbrace{1 \dots 1}_{N_1} \underbrace{2 \dots 2}_{N_2}} (y).$$
The hypothesis $\mu h = 0$ gives 
\begin{equation} \label{lemma1_0}
h_{2 m_2 \dots m_k 2 n_2 \dots n_\ell} (y) = - \hspace{1mm} h_{1 m_2 \dots m_k 1 n_2 \dots n_\ell} (y).
\end{equation}
Let $M = \min \{M_2, N_2\}$ then by applying this identity $M$ times, we get
\begin{equation} \label{lemma1_1}
h_{m_1 \dots m_k n_1 \dots n_\ell} (y) = {(-1)}^M h_{\underbrace{1 \dots 1}_{M_1 + M} \underbrace{2 \dots 2}_{M_2 - M} \underbrace{1 \dots 1}_{N_1 + M} \underbrace{2 \dots 2}_{N_2 - M}} (y).
\end{equation}
Depending on whether $M  = M_2 \mbox{ or } N_2$, we can further simplify this expression in the following way.\smallskip

 \textbf{Case 1:} $M = N_2$.  Inserting $M = N_2$ in equation \eqref{lemma1_1}, we deduce
\begin{equation}
h_{m_1 \dots m_k n_1 \dots n_\ell} (y) = {(-1)}^{N_2} h_{\underbrace{1 \dots 1}_{M_1 + N_2} \underbrace{2 \dots 2}_{M_2 - N_2} \underbrace{1 \dots 1}_{\ell}}(y).
\end{equation}
Also, from the other hypotheses of the lemma, we have $y_m h_{m m_2 \dots m_k n_1 \dots n_\ell} = 0$, that is
\begin{equation*}
h_{2 m_2 \dots m_k n_1 \dots n_\ell} (y) = - \hspace{1mm} \frac{y_1}{y_2} \hspace{1mm} h_{1 m_2 \dots m_k n_1 \dots n_\ell} (y).
\end{equation*}
This finally gives the following equation:
\begin{align*}
h_{m_1 \dots m_k n_1 \dots n_\ell} (y) &= {(-1)}^{N_2} \left\{- \hspace{1mm} \frac{y_1}{y_2}\right\}^{M_2 - N_2} h_{\underbrace{1 \dots 1}_k \underbrace{1 \dots 1}_\ell}(y)\\
&= {(-1)}^{M_2} \left(\frac{y_1}{y_2}\right)^{M_2 - N_2} h_{\underbrace{1 \dots 1}_k \underbrace{1 \dots 1}_\ell}(y).
\end{align*}
\smallskip

\textbf{Case 2:} $M= M_2$.

 In this case, equation \eqref{lemma1_1} becomes
\begin{equation} \label{lemma1_2}
h_{m_1 \dots m_k n_1 \dots n_\ell} (y) = {(-1)}^{M_2} h_{\underbrace{1 \dots 1}_{k} \underbrace{1 \dots 1}_{N_1 + M_2} \underbrace{2 \dots 2}_{N_2 - M_2}}(y).
\end{equation}
Again, using $y_m h_{m m_2 \dots m_k n_1 \dots n_\ell} = 0$, we have 
\begin{equation*}
h_{1 m_2 \dots m_k n_1 \dots n_\ell} (y) = - \hspace{1mm} \frac{y_2}{y_1} \hspace{1mm} h_{2 m_2 \dots m_k n_1 \dots n_\ell} (y).
\end{equation*}
This converts equation \eqref{lemma1_2} to the following equation:
\begin{equation}
h_{m_1 \dots m_k n_1 \dots n_\ell} (y) = {(-1)}^{M_2} \left\{- \hspace{1mm} \frac{y_2}{y_1}\right\}^{N_2 - M_2} h_{\underbrace{1 \dots 1}_{k - \left(N_2 - M_2\right)} \underbrace{2 \dots 2}_{N_2 - M_2} \underbrace{1 \dots 1}_{N_1 + M_2} \underbrace{2 \dots 2}_{N_2 - M_2}}.
\end{equation}
Finally, using \eqref{lemma1_0}, we conclude 
\begin{align*}
h_{m_1 \dots m_k n_1 \dots n_\ell} (y) &= {(-1)}^{M_2} \left\{- \hspace{1mm} \frac{y_2}{y_1}\right\}^{N_2 - M_2} {(-1)}^{N_2 - M_2} h_{\underbrace{1 \dots 1}_k \underbrace{1 \dots 1}_\ell}(y)\\
&= {(-1)}^{M_2} {\left(\frac{y_1}{y_2}\right)}^{M_2 - N_2}h_{\underbrace{1 \dots 1}_k \underbrace{1 \dots 1}_\ell}(y)\\
&= {(-1)}^{k - N(m_1,\dots,m_k)} \left(\frac{y_1}{y_2}\right)^{\left(k - \ell) - (N(m_1,\dots,m_k) - N(n_1,\dots,n_\ell)\right) } h_{\underbrace{1 \dots 1}_k \underbrace{1 \dots 1}_\ell} (y),
\end{align*}
where in the last equality, we use definitions of $M_2$ and $N_2$. This completes the proof of the lemma. 
\end{proof}
\noindent In light of this lemma, we see that to prove Theorem \ref{th:inversion_normal_ope}, it is sufficient to express $\widehat{f^s}_{\underbrace{1 \dots 1}_k \underbrace{1 \dots 1}_\ell}$ in terms of the Fourier transform of the normal operator of MiRT. Now we are ready to present the proof of the Theorem \ref{th:inversion_normal_ope}. 
\begin{proof}[Proof of Theorem \ref{th:inversion_normal_ope}]

The aim here is to recover $f^s$ componentwise from the knowledge of the normal operator $\textit{N}_{k,\ell} f$ of the MiRT. Let us recall the following identity, proved in Lemma \ref{lem;normal_operator}: 
$$({\textit{N}_{k,\ell} f})_{i_1 \dots i_k j_1 \dots j_\ell} (x) = 2 f_{m_1 \dots m_k n_1 \dots n_\ell} (x)  *  {(- 1)}^{(J_1 + N_1)} \quad \frac{x_1^{(I_1 + J_2 + M_1 + N_2)} x_2^{(I_2 + J_1 + M_2 + N_1)}}{|x|^{2k + 2\ell+ 1}},$$
where $I_i, J_i, M_i$, and $N_i$ are functions of corresponding indices $(i_1, \dots, i_k), (j_1, \dots, j_\ell), (m_1, \dots, m_k)$, and $(n_1, \dots, n_\ell)$ respectively (defined earlier as in Lemma \ref{lem;normal_operator}).  
The idea here is to take the Fourier transform of this equation and simplify the right-hand side of the obtained equation by using the properties of the Fourier transform and convolution. This will allow us to write the components of the Fourier transform of $f^s$ in terms of the Fourier transform of the normal operator $\textit{N}_{k,\ell} f$. Finally, $f^s$ is obtained componentwise using Fourier inversion.

As mentioned above, taking the Fourier transform to equation \eqref{Normal operator MRT} and using the relation $\widehat{u \ast v} = (2 \pi) \widehat{u} \widehat{v}$ and $ \textit{N}_{k,\ell} f = \textit{N}_{k,\ell} f^s $, we obtain
\begin{align*}
    \F\left[({\textit{N}_{k,\ell} f})_{i_1 \dots i_k j_1 \dots j_\ell}\right] &= \F \left[({\textit{N}_{k,\ell} f^s})_{i_1 \dots i_k j_1 \dots j_\ell}\right]\\ 
    &= {(- 1)}^{(J_1 + N_1)} \left(4 \pi\right) \widehat{f^s}_{m_1 \dots m_k n_1 \dots n_\ell} \F \left[\frac{x_1^{(I_1 + J_2 + M_1 + N_2)} x_2^{(I_2 + J_1 + M_2 + N_1)}}{|x|^{2k + 2\ell + 1}}\right].
\end{align*}
The right-hand side of the above equation is the product of a smooth function and a tempered distribution. Further calculating the second factor on the right-hand side in the above equation using the properties of the Fourier transform, we get
\begin{equation}
    \begin{aligned} \label{fourier transform of normal operator}
    \F \left[({\textit{N}_{k,\ell} f})_{i_1 \dots i_k j_1 \dots j_\ell}\right] (y) &= \frac{{(-1)}^{k + \ell + J_1 + N_1} \pi}{2^{2{(k + \ell - 1)}}} \frac{\Gamma\left(- (k + \ell) - \frac{1}{2}\right)}{\Gamma \left(k + \ell + \frac{1}{2}\right)} \\ 
    &\qquad\times \widehat{f^s}_{m_1 \dots m_k n_1 \dots n_\ell} (y) 
    \,\,\partial_{\underbrace{1 \dots 1}_{A \text{ times}} \underbrace{2 \dots 2}_{B \text{ times}}} |y|^{2{(k + \ell)} - 1},
\end{aligned}
\end{equation}
where $A = I_1 + J_2 + M_1 + N_2,$ \hspace{1mm}$ B = I_2 + J_1 + M_2 + N_1$ and $\displaystyle \partial_{i_1 \dots i_m} = \frac{\partial^{m}}{\partial y_{i_1} \dots \partial y_{i_m}}$. 

To proceed further, we need to simplify the last term $ \displaystyle \partial_{\underbrace{1 \dots 1}_{A \text{ times}} \underbrace{2 \dots 2}_{B \text{ times}}} |y|^{2{(k + \ell)} - 1}$ appearing in the above equation. 
For this purpose, let us define $e_i = y_i / |y| \in C^\infty (\FR^2 \symbol{92} \left\{0\right\})$ and a tensor field $\varepsilon_{ij}(y) = \delta_{ij} - e_i(y)e_j(y)$ as in \cite{Sharafutdinov_1994} then we have the following identity  \cite[Lemma 2.11.1]{Sharafutdinov_1994}:

\begin{equation*}
    \partial_{i_1 \dots i_{2m}} |y|^{2m - 1} = {((2m - 1)!!)}^2 |y|^{-1} \sigma (\varepsilon_{i_1i_2} \dots \varepsilon_{i_{2m - 1}i_{2m}}),
\end{equation*}
where $\displaystyle k !! = k (k - 2) (k - 4) \dots$ and $(- 1) !! = 1$. In our setup, the above equality takes the following form:
\begin{equation} \label{partial derivative of |y|}
\begin{aligned}
 \partial_{\underbrace{1 \dots 1}_{A \text{ times}} \underbrace{2 \dots 2}_{B \text{ times}}} |y|^{2{(k + \ell)} - 1}&= {(- 1)}^{(I_1 + J_2 + M_1 + N_2)} {((2(k + \ell) - 1)!!)}^2\\& \qquad\frac{y_1^{(I_2 + J_1 + M_2 + N_1)}y_2^{(I_1 + J_2 + M_1 + N_2)}}{|y|^{2 (k + \ell) + 1}}.
\end{aligned}
\end{equation}
Inserting this expression in equation \eqref{fourier transform of normal operator}, we conclude
\begin{align*} 
\F\left[({\textit{N}_{k,\ell} f})_{i_1 \dots i_k j_1 \dots j_\ell}\right] (y) &= b (k, \ell) {(- 1)}^{(I_1 + M_1)} \widehat{f^s}_{m_1 \dots m_k n_1 \dots n_\ell} (y) \frac{y_1^{(I_2 + J_1 + M_2 + N_1)}y_2^{(I_1 + J_2 + M_1 + N_2)}}{|y|^{2 (k + \ell) + 1}}\\
&= b (k, \ell) {(- 1)}^{I_1} \frac{y_1^{I_2 + J_1}y_2^{I_1 + J_2}}{|y|^{2 (k + \ell) + 1}} \left\{{(- 1)}^{M_1} y_1^{M_2 + N_1} y_2^{M_1 + N_2} \widehat{f^s}_{m_1 \dots m_k n_1 \dots n_\ell} (y)\right\},
\end{align*}
where the constant $b_{k,\ell}$ is given by
\begin{align*}
    b (k, \ell) = \frac{{(- 1)}^{k + \ell} {((2(k + \ell) - 1)!!)}^2}{2^{2{(k + \ell - 1)}}} \frac{\pi \Gamma\left(- (k + \ell) - \frac{1}{2}\right)}{\Gamma \left(k + \ell + \frac{1}{2}\right)}.
\end{align*}

Note that, there is a summation over indices $m_1, \dots, m_k$ and $n_, \dots, n_\ell$ appearing in $M_1$, $M_2$ and $N_1, N_2$ on the right-hand side of the above equation. From Lemma \ref{lem: relation between components of f}, to recover $f^s$ it is sufficient to know $f^s_{\underbrace{1 \dots 1}_k \underbrace{1 \dots 1}_\ell}$. With this motivation, we take $i_1 = \dots = i_k = 1 = j_1 = \dots = j_\ell$ in the above relation to get

\begin{align*}
\F\left[({\textit{N}_{k,\ell} f})_{\underbrace{1 \dots 1}_{k} \underbrace{1 \dots 1}_{\ell}}\right] (y) &= b (k, \ell) \frac{ (- 1)^k \,y_1^\ell y_2^k}{|y|^{2 (k + \ell) + 1}} \left\{{(- 1)}^{M_1} y_1^{M_2 + N_1} y_2^{M_1 + N_2} \widehat{f^s}_{m_1 \dots m_k n_1 \dots n_\ell} (y)\right\}\\
&= \frac{ b (k, \ell) }{|y|^{2 (k + \ell) + 1}} \left\{ 
y_1^{l + 2M_2 + N_1 - N_2} y_2^{k + 2N_2 + M_1 - M_2} \widehat{f^s}_{1 \dots 1 1 \dots 1} (y)\right\}  \text{ (Lemma \ref{lem: relation between components of f})}\\
&= \frac{ b (k, \ell) }{|y|^{2 (k + \ell) + 1}} \widehat{f^s}_{1 \dots 1 1 \dots 1} (y) \left\{\sum_{r = 0}^{k + \ell} \binom{k + \ell}{r} {({y_1}^2)}^{k + \ell - r} {({y_2}^2)}^{r} \right\}\\
&= \frac{ b (k, \ell) }{|y|^{2 (k + \ell) + 1}} \widehat{f^s}_{1 \dots 1 1 \dots 1} (y) {|y|}^{2(k + \ell)}\\
&= \frac{ b (k, \ell) }{|y|} \widehat{f^s}_{1 \dots 1 1 \dots 1} (y).
\end{align*}

Hence we have
\begin{equation} \label{Fourier Transform of first component}
    \widehat{f^s}_{1 \dots 1 1 \dots 1} (y) = \frac{|y|}{b (k, \ell)} \F \left[({\textit{N}_{k,\ell} f})_{\underbrace{1 \dots 1}_{k} \underbrace{1 \dots 1}_{\ell}}\right] (y). 
\end{equation}
It is known that $|y| \F (\cdot) = \F {(- \Delta)}^{1/2}(\cdot)$, where ${(- \Delta)}^{1/2} u = - \frac{1}{2 \pi} \hspace{1mm} u \ast {|x|}^{- 3}$. This implies
\begin{equation*}
    {f^s}_{1 \dots 1 1 \dots 1} = \frac{1}{b (k, \ell)} {(- \Delta)}^{1/2} ({\textit{N}_{k,\ell} f})_{1 \dots 1 1 \dots 1}.
\end{equation*}
Next combining Lemma \ref{lem: relation between components of f}, and equation \eqref{Fourier Transform of first component}, we derive
\begin{equation*}
    \widehat{f^s}_{m_1 \dots m_k n_1 \dots n_\ell} (y) = {(- 1)}^{M_2} {\left(\frac{y_1}{y_2}\right)}^{M_2 - N_2} \frac{|y|}{b (k, \ell)} \F\left[({\textit{N}_{k,\ell} f})_{1 \dots 1 1 \dots 1}\right] (y),
\end{equation*}
which is the required relation.  Further simplifying the above formula, we obtain 
\begin{align*}
    y^{2N_2}_2  \widehat{f^s}_{m_1 \dots m_k n_1 \dots n_\ell} (y) = {(- 1)}^{M_2}  y_1^{M_2}\, y_2^{N_2}  \frac{|y|}{b (k, \ell)} \F \left[({\textit{N}_{k,\ell} f})_{1 \dots 1 1 \dots 1}\right] (y).
\end{align*}
This implies 
\begin{align*}
   \F( (\Delta')^{N_2} f^s)_{m_1 \dots m_k n_1 \dots n_\ell}= (-1)^{M_2+N_2}\, \F\left( (-\Delta)^{1/2} P(D) \,  (N_{k,\ell}f)_{1\dots 11\dots 1}\right),
\end{align*}
where $ \Delta'= \partial^2_{x_2}$ is the Laplace operator in $x_2$ variable and $P(D) = D^{M_2}_{x_1} \, D^{N_2}_{x_2}$, where $ D= \frac{1}{\I }\partial$, is a constant coefficient differential operator whose symbol is $y^{M_2}_1\, y^{N_2}_2$. Therefore we have 
\begin{align}
    f^s_{m_1 \dots m_k n_1 \dots n_\ell} = (-1)^{M_2}\, (-\Delta)^{1/2}(-\Delta')^{-N_2} P(D) \,  (N_{k,\ell}f)_{1\dots 11\dots 1}.
\end{align}
This completes the proof.
\end{proof}

\subsection{Unique continuation property} \label{sec: UCP}
This section is devoted to proving a UCP result for MiRT. The idea here is to relate MiRT data to LRT data (for appropriately modified tensor field) and then use the known UCP results for LRT.

\begin{proof}[Proof of Theorem \ref{th:ucp_MiRT}]
We prove this theorem for $f \in C_c^\infty(\mathit{S}^k \times \mathit{S}^\ell)$. Let $I_mg$ be the longitudinal ray transform of $g$ as defined in Definition \eqref{def_of_Im}. Then, it is known that the MiRT of $f$ is related to the LRT of $g = \sigma A f$ in the following way (see \cite[Equation 20]{kernel_MRT} for detailed discussion):
\begin{align} \label{relation f and g}
\mathit{L}_{k, \ell}f = I_{k + \ell}g \implies \mathit{L}_{k, \ell}f = 0 \iff I_{k + \ell}g = 0.
\end{align}

This identity, together with the facts $I_m g = 0$ if and only if $g  = \D v$ (see \cite[Theorem 2.2.1]{Sharafutdinov_1994}) and $\mathit{L}_{k, \ell}f = 0$ if and only if  $f = \D^\prime u + \lambda w$ (see  \cite[Theorem 1]{kernel_MRT}) gives 
\begin{align}\label{kernel of MRT for f wrt g}
    f = \D^\prime u + \lambda w \iff \mathit{L}_{k, \ell}f = 0 \iff I_{k + \ell}g = 0 \iff g = \D v
\end{align} 
for some $v \in C^\infty (\mathit{S}^{k + \ell - 1})$, $u \in C^\infty (\mathit{S}^{k-1} \times \mathit{S}^{\ell}), w \in C_c^\infty (\mathit{S}^{k - 1} \times \mathit{S}^{\ell - 1})$ and $g = \sigma A f $ as defined previously. 
 Thus, we have $  f = \D^\prime u + \lambda w \iff g = \D v.$
 
To complete the proof of the theorem, it is sufficient to show that $g = \D v$ for some $v \in C^\infty (\mathit{S}^{k + \ell - 1})$ and $g = \sigma \textit{A}f$. Note $g|_U = 0$ because $f|_U = \lambda w$ and $I_{k + \ell}g(x, \xi) = 0$ for all $(x,\xi) \in U\times \Sb^{n-1}$ (from equation \eqref{relation f and g}). Thus, $g$ satisfies the assumptions of the \cite{UCP_2022}*{Theorem 2.2}, and hence, by applying this theorem, we get $g = \D v$ for some $v$ as required. This completes the proof of the Theorem \ref{th:ucp_MiRT}.
\end{proof}
We now proceed proving UCP using the normal operator of MiRT. To this end, we next  establish a relation  between the normal operator of $L_{k,\ell} f$, and that of $I_{k+\ell} g$, where $g =\sigma Af$.
\begin{lemma} \label{lem: relationship between normal operators}
    Let $f \in C_c^\infty(\mathit{S}^k \times \mathit{S}^\ell)$ and $g = \sigma A f$. Then the normal operators $\textit{N} g$ and $\textit{N}_{k, l} f$ are related by the following relation:
    \begin{equation} \label{eq: relation between normal operators}
        (\textit{N} g)_{i_1 \dots i_k j_1 \dots j_\ell} (x) = {(- 1)}^{J_2} (\textit{N}_{k, \ell} f)_{i_1 \dots i_k \delta(j_1, \dots, j_\ell)} (x).
    \end{equation}
\end{lemma}
\begin{proof}
    We have
    \begin{equation*}
        g_{i_1 \dots i_k j_1 \dots j_\ell} (x) = \sigma (i_1, \dots, i_k, j_1, \dots, j_\ell) {(-1)}^{J_2} f_{i_1 \dots i_k \delta(j_1, \dots, j_\ell)} (x).
    \end{equation*}
Hence for any $0 \leq p \leq k + \ell$, we obtain
    \begin{equation} \label{eq: g and f relation with indices}
        g_{\underbrace{1 \dots 1}_{p} \underbrace{2 \dots 2}_{k + \ell - p}} = \frac{1}{(k + \ell)!} \left\{p! (k + \ell - p)! \left[\sum_{\substack{r + s = p \\ 0 \leq r \leq k \\ 0 \leq s \leq \ell}} \binom{k}{r} \binom{\ell}{s} {(- 1)}^{\ell - s} f_{\underbrace{1 \dots 1}_{r} \underbrace{2 \dots 2}_{k - r} \underbrace{1 \dots 1}_{\ell - s} \underbrace{2 \dots 2}_{s}} \right]\right\}. 
    \end{equation}
   
    Keeping the above equation in mind, we start with the left-hand side of equation \eqref{eq: relation between normal operators}:
    \begin{align*}
        (\textit{N}& g)_{i_1 \dots i_k j_1 \dots j_\ell} (x) \\&= 2 \hspace{1mm} g_{m_1 \dots m_k n_1 \dots n_\ell} (x) \ast \frac{{x_1}^{I_1 + J_1 + M_1 + N_1} {x_2}^{I_2 + J_2 + M_2 + N_2}}{|x|^{2k + 2\ell+ 1}}\\
        &= 2 \sum_{p = 0}^{k + l} \binom{k + l}{p} g_{\underbrace{1 \dots 1}_{p} \underbrace{2 \dots 2}_{k + \ell - p}} (x) \ast \frac{{x_1}^{p + I_1 + J_1} {x_2}^{k + \ell - p + I_2 + J_2}}{|x|^{2k + 2\ell+ 1}}\\
        &= 2 \sum_{p = 0}^{k + l} \left[\sum_{\substack{r + s = p \\ 0 \leq r \leq k \\ 0 \leq s \leq \ell}} \binom{k}{r} \binom{\ell}{s} {(- 1)}^{\ell - s} f_{\underbrace{1 \dots 1}_{r} \underbrace{2 \dots 2}_{k - r} \underbrace{1 \dots 1}_{\ell - s} \underbrace{2 \dots 2}_{s}} (x) \right] \ast \frac{{x_1}^{p + I_1 + J_1} {x_2}^{k + \ell - p + I_2 + J_2}}{|x|^{2k + 2\ell+ 1}}\\
        &= 2 \sum_{\substack{r + s = 0 \\ 0 \leq r \leq k \\ 0 \leq s \leq \ell}}^{k + \ell} \binom{k}{r} \binom{\ell}{s} {(- 1)}^{\ell - s} f_{\underbrace{1 \dots 1}_{r} \underbrace{2 \dots 2}_{k - r} \underbrace{1 \dots 1}_{\ell - s} \underbrace{2 \dots 2}_{s}} (x) \ast \frac{{x_1}^{r + s + I_1 + J_1} {x_2}^{k + \ell - (r + s) + I_2 + J_2}}{|x|^{2k + 2\ell+ 1}}\\
        &= 2 \sum_{r = 0}^k \sum_{s = 0}^\ell \binom{k}{r} \binom{\ell}{s} {(- 1)}^{\ell - s} f_{\underbrace{1 \dots 1}_{r} \underbrace{2 \dots 2}_{k - r} \underbrace{1 \dots 1}_{\ell - s} \underbrace{2 \dots 2}_{s}} (x) \ast \frac{{x_1}^{r + s + I_1 + J_1} {x_2}^{k + \ell - (r + s) + I_2 + J_2}}{|x|^{2k + 2\ell+ 1}}\\
        &= 2 f_{m_1 \dots m_k n_1 \dots n_\ell} (x) \ast {(- 1)}^{N_1} \frac{{x_1}^{I_1 + J_1 + M_1 + N_2} {x_2}^{I_2 + J_2 + M_2 + N_1}}{|x|^{2k + 2\ell+ 1}}\\
        &= 2 {(- 1)}^{J_2} f_{m_1 \dots m_k n_1 \dots n_\ell} (x) \ast {(- 1)}^{J_2 + N_1} \frac{{x_1}^{I_1 + J_1 + M_1 + N_2} {x_2}^{I_2 + J_2 + M_2 + N_1}}{|x|^{2k + 2\ell+ 1}}\\
        &= {(- 1)}^{J_2} (\textit{N}_{k, \ell} f)_{i_1 \dots i_k \delta(j_1, \dots, j_\ell)} (x)\\
        &= A (N_{k,\ell}f)_{i_1 \dots i_k j_1 \dots j_\ell}(x).
    \end{align*}
    This finishes the proof.
\end{proof}
\begin{proof}[Proof of Theorem \ref{th_ucp_normal_op} and Theorem \ref{th_ucp_MiRT_1}]
We first show that $N_{k,\ell} f $ is smooth in $U$ to make sense of the assumption that $N_{k,\ell} f$ vanishes to infinite order at $x_0\in U$. By Lemma \ref{lem: relationship between normal operators} we have that $N_{k,\ell} f= Ng $, where $g= \sigma Af$. Since $f= \lambda w+ P(x)$ in $U$, this implies $ g= 0$ in $U$. Therefore by \cite{UCP_2022}*{Lemma 3.1} we have that $Ng$ is smooth in $U$, which further entails that $N_{k,\ell}f $ is smooth in $U$. This justifies the assumption $N_{k,\ell} f$ vanishes to infinite order at some $x_0\in U$.

Next, we have that the tensor field $g= \sigma A f$ and $Ng$ satisfies the assumption of Theorem $2.1$ from \cite{UCP_2022}.  This implies $ g=dv$ for some $v \in \E'(S^{k+\ell-1})$. Next, using the relation 
\[   f= \D 'u +\lambda w \iff g= \D v\] we conclude the proof of Theorem \ref{th_ucp_normal_op}.

To prove Theorem \ref{th_ucp_MiRT_1} we observe that if $f =\D' u +\lambda w'$, then 
\begin{align}
    g= \sigma A f = \sigma A (\D'u +\lambda w')= \D v.
\end{align}
To obtain the above relation, we have used the fact that $ \sigma A(\lambda w')=0$ and $A$ commutes with $\D'$. We have that $Ng$ vanishes to infinite order at $x_0\in U$ and $Rg=0$ in $U$, where $R$ is the Saint-Venant operator. Therefore, by \cite{UCP_2022}*{Theorem 2.1} we have that $g= \D v$. This completes the proof of Theorem \ref{th_ucp_MiRT_1}. 
\end{proof}

\subsection{Range characterization}\label{sec:range_characterization}
In this section, we characterize the range of the mixed ray transform acting on symmetric tensor fields belonging to $\mathcal{S}(S^k\times S^{\ell})$ in two dimensions. We use the result available for the ray transform acting on symmetric tensor fields in two dimensions: \cite{Pantyukhina_range_2d}.  The presentation and notation in this section are based on \cite{Momentum_ray_transform_2_2020}*{Section 3}. 

 Let $f$ be a symmetric $m$ tensor field in $\rt$. Then $f$ can be uniquely written as:
 \begin{align}
     f(x_1,x_2)= \sum_{j=0}^m \Tilde{f}_j(x_1,x_2)\, d x_1^{m-j}\, d x_2^j,
 \end{align}
where $\Tilde{f}_j=  \binom{m}{j}f_{\underbrace{1\cdots 1}_{m-j} \underbrace{ 2\cdots 2}_{j}} $ for $0\le j\le m.$
\newcommand{\C}{\mathbb{C}}
We  next identity $\rt$ with the complex plane $\mathbb{C}=\{ z: z= x_1+\I x_2 \}$, where $ \I= \sqrt{-1}$ is the imaginary unit. We also write a point $ (x,\xi) \in \rt\times \rt \setminus\{0\}$ as $(z,\zeta) \in \C \times \C\setminus\{0\} $ where $ z= x_1+\I x_2$ and $\zeta = \xi_1+\I \xi_2$. Then using the relations \cite{Momentum_ray_transform_2_2020}*{Equations $3.2-3.4$} we can write the ray transform of $f\in \mathcal{S}(S^m)$ 
as 
\begin{align}
   J: \mathcal{S}(S^m) \mapsto C^{\infty} (\rt\times \rt \setminus\{0\}),\quad   Jf(z,\zeta) =\int_{\R} \sum_{j=0}^m\Tilde{f}_j (z+t\zeta) \, \xi_1^{m-j}\, \xi_2^{j} \, dt.
\end{align}
Similarly we define the mixed ray transform of $f\in \mathcal{S}(S^k\times S^{\ell})$ in the following way
\begin{align}
    L_{k,\ell} f(z,\zeta) = Jg (z,\zeta), \quad \mbox{where $ g=\sigma A f$}
\end{align}
where the operators  $\sigma$ and $A$ are given by \eqref{def_of_sigma}, and \eqref{def_of_A} respectively.

\begin{proof}[Proof of Theorem \ref{th:range_characterization}]
    Suppose, there is a $f\in  \mathcal{S}(S^k\times S^{\ell})$ such that $ \phi = L_{k,\ell} f$, then using the relation $   L_{k,\ell} f(z,\zeta) = Jg (z,\zeta)$, where $ g=\sigma A f$  one can easily  verify conditions $[(i)-(ii)]$ similar to the case of the ray transform. 

Next we assume that $\phi$ satisfies $[(i)-(ii)]$. Since this coincides with the range characterization of the ray transform of a symmetric $(k+\ell)$ tensor fields in $\rt$, there is a $g \in \mathcal{S}(\rt; S^{k+\ell})$ such that 
\begin{align}
I_{k+\ell}g = \phi. 
\end{align}
Given such a $g\in  \mathcal{S}(\rt; S^{k+\ell}) $ we need to find a $f \in  \mathcal{S}(\rt; S^{k} \times S^{\ell})$ such that $ L_{k,\ell} f =\phi$. According to the kernel description, $f $ is unique modulo a term of the form $ \D'v+ \lambda w$, where the definition of  $\D'$ and $\lambda$ are given in \eqref{def: d' operator} and \eqref{def: lambda operator}. This can be done by repeating the proof of Theorem \ref{thm: tensor field decomposition}. We next define $ f$ in the following way:
\begin{align}
f = A^{-1} g\implies Af =g.
\end{align} 
 
Therefore, 
\begin{align}
L_{k,\ell} f = I_{k+\ell}( \sigma A f ) = I_{k+\ell}( \sigma A (A^{-1} g))= I_{k+\ell} (\sigma g) = I_{k+\ell}g= \phi.
\end{align} 
Note that the tensor field $f$ is not unique. This follows from combining the relation $\sigma A f= g$  and the proof of Theorem \ref{thm: tensor field decomposition}. To finish the proof, we have to show that $ Af$ is a symmetric $(k+\ell)$ tensor field. 

If $g$ is a $(k+\ell)$-symmetric tensor field, then we show that $f = A^{-1} g$ is trace-free $i.e.,$ $f$ satisfies $\mu f=0$.

From \cite[Equation 18]{kernel_MRT}, we have 
\begin{align*}
A^{-1} = {(- 1)}^\ell A \D \implies f_{i_1 \dots i_k j_1 \dots j_\ell} = {(A^{-1} g)}_{i_1 \dots i_k j_1 \dots j_\ell} = {(- 1)}^{N(j_1, \dots, j_\ell)} g_{i_1 \dots i_k \delta(j_1, \dots, j_\ell)}.
\end{align*}
This implies
   
    \begin{align*}
        f_{1 i_2 \dots i_k 1 j_2 \dots j_\ell} &= {(- 1)}^{N(1, j_2, \dots, j_\ell)} g_{1 i_2 \dots i_k 2 \delta(j_2, \dots, j_\ell)}\\
        &= - {(- 1)}^{N(2, j_2, \dots, j_\ell)} g_{2 i_2 \dots i_k 1 \delta(j_2, \dots, j_\ell)}\\
        &= - f_{2 i_2 \dots i_k 2 j_2 \dots j_\ell}.
    \end{align*}
    
Observe that $A f$ is symmetric over the first $k$ and last $\ell$ indices. Next using the fact that $\mu f=0$, we show that $A f$ is symmetric over all $(k + \ell)$ indices.  Now, to prove symmetry over all indices, it is sufficient to show the following:
        $$(Af)_{i_1 \dots i_k j_1 \dots j_\ell} = (Af)_{j_1 i_2 \dots i_k i_1 j_2 \dots j_\ell}.$$
        
        The above relation can be proved as follows: If $i_1 = j_1$, then the above equality is trivial. Let $i_1 \neq j_1$. This implies $i_1 = \delta(j_1)$ and $j_1 = \delta(i_1)$ (this holds true only in two dimensions). Now,
        \begin{align*}
            (Af)_{i_1 \dots i_k j_1 \dots j_\ell} &= {(- 1)}^{\ell - N(j_1, \dots, j_\ell)} f_{i_1 \dots i_k \delta(j_1, \dots, j_\ell)}\\
            &= - {(- 1)}^{\ell - N(\delta(j_1), j_2,\dots, j_\ell)} f_{i_1 \dots i_k \delta(j_1)\delta(j_2, \dots, j_\ell)}\\
            &= - {(- 1)}^{\ell - N(i_1, j_2, \dots, j_\ell)} f_{i_1 \dots i_k \delta(j_1)\delta(j_2, \dots, j_\ell)}\\
            &= {(- 1)}^{\ell - N(i_1, j_2, \dots, j_\ell)} f_{\delta(i_1) i_2 \dots i_k j_1\delta(j_2, \dots, j_\ell)}\\
            &= {(- 1)}^{\ell - N(i_1, j_2, \dots, j_\ell)} f_{j_1 i_2 \dots i_k \delta(i_1, j_2, \dots, j_\ell)} = (Af)_{j_1 i_2 \dots i_k i_1 j_2 \dots j_\ell}.
        \end{align*}
        In the above computation, since $i_1 = \delta(j_1)$ and $f$ is trace-free, we have used the following relation: 
        $$f_{i_1 i_2 \dots i_k \delta(j_1) j_2 \dots j_\ell} = - f_{\delta(i_1) i_2 \dots i_k \delta(\delta(j_1)) j_2 \dots j_\ell} = - f_{\delta(i_1) i_2 \dots i_k j_1 j_2 \dots j_\ell}.$$
        This completes the proof.    
    \end{proof}

\appendix
\section{Motivation behind Lemma \ref{lem: relation between components of f}: a pictorial illustration}
The number of independent components of the tensor $f \in C^\infty (\mathit{S}^k \times \mathit{S}^\ell)$ is $\binom{n + k - 1}{k} \times \binom{n + \ell - 1}{\ell}$, that is, in our case, $(k + 1)(\ell + 1)$. Our goal in lemma \ref{lem: relation between components of f} is to represent a component $f_{i_1 \dots i_k j_1 \dots j_\ell}$ of a $m =  k + \ell$ tensor field $f$ in terms of a single component $\displaystyle f_{\underbrace{1\dots 1}_{k} \underbrace{1\dots 1}_{\ell}}$. For this, we need $(k + 1)(\ell + 1) - 1 = k\ell + k + \ell$ independent conditions. The condition $\mu f = 0$ ($\delta^{ij} f_{i i_2 \dots i_k j j_2 \dots j_\ell} = 0$) gives $k \ell$ conditions and the condition $y_i f_{i i_2 \dots i_k j_1 \dots j_\ell} = 0$ gives another $k (\ell + 1)$ conditions. So in total, we have $k\ell + k (\ell + 1)$ conditions between the components of $f$, which is $\ell (k - 1)$ more than the required conditions. This gives a hint that it might be possible to relate each component of $f$ to a single component, which we already proved analytically on page 10. The following graphs give an alternate way to see the same result. We present these graphs for $m=5$ and different values of $k$ and $\ell$ such that $m = k + \ell = 5$.
\begin{figure}[H]
\centering
\begin{subfigure}[c]{0.48\textwidth}
\centering
\includegraphics[width=\textwidth]{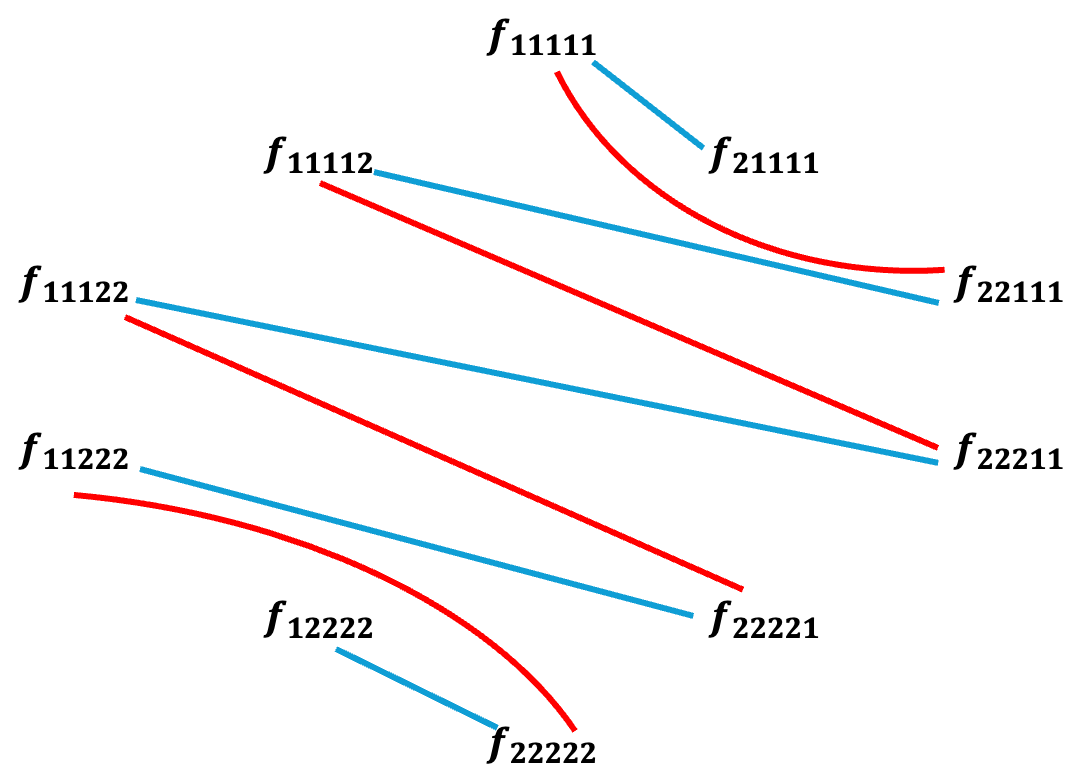}\caption{$m = 5$, $k = 1$, and $\ell = 4$}\label{fig:}
\end{subfigure}
\hfill
\begin{subfigure}[c]{0.48\textwidth}
\centering
\includegraphics[width=\textwidth]{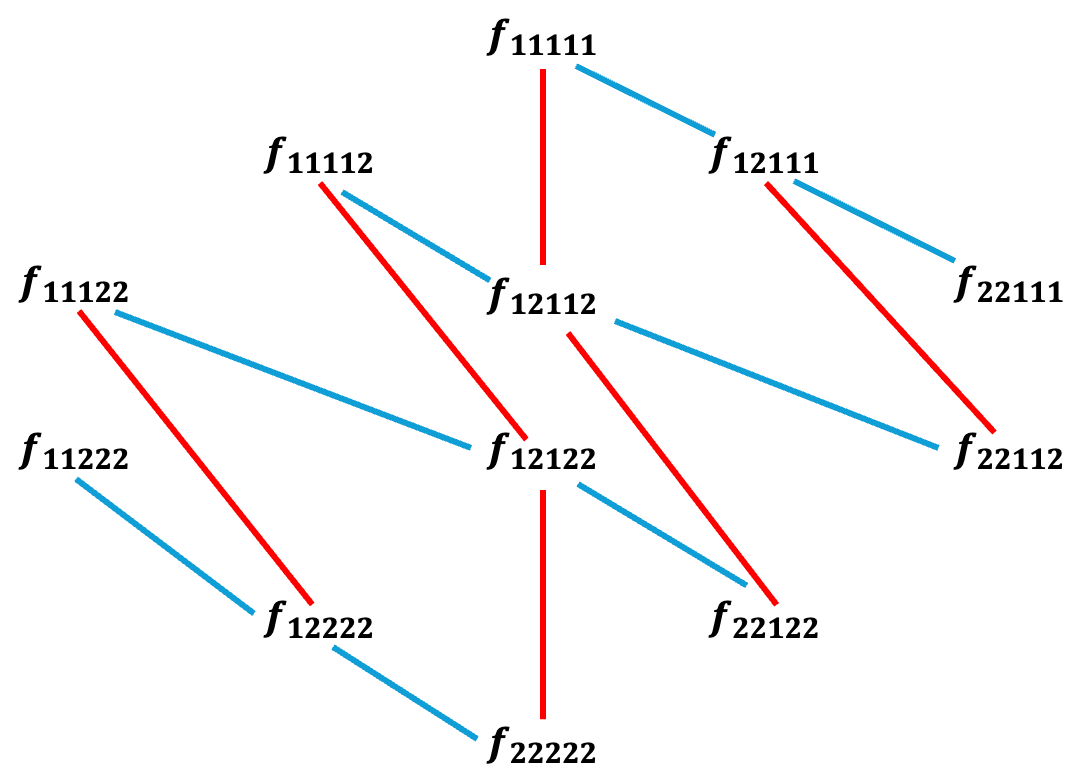}\caption{$m = 5$, $k = 4$, and $\ell = 1$}\label{fig:ss}
\end{subfigure}
\vfill
\centering
\begin{subfigure}[c]{0.48\textwidth}
\centering
\includegraphics[width=\textwidth]{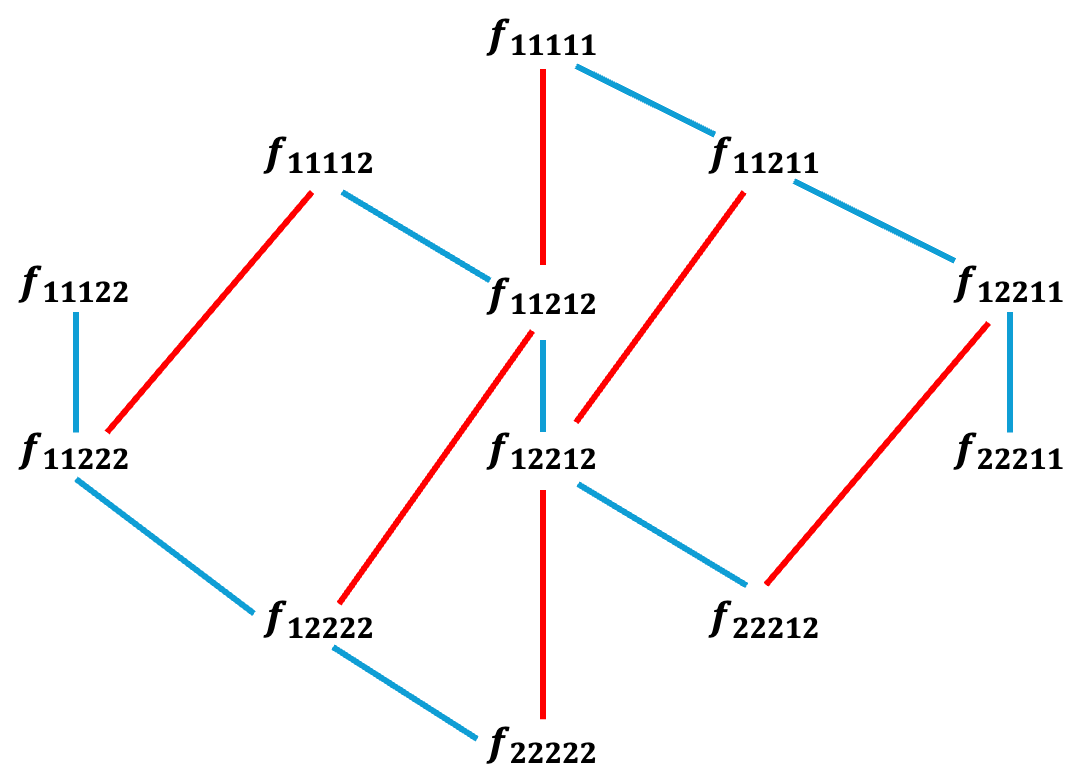}\caption{$m = 5$, $k = 2$, and $\ell = 3$}\label{fig: def of}
\end{subfigure}
\hfill
\begin{subfigure}[c]{0.48\textwidth}
\centering
\includegraphics[width=\textwidth]{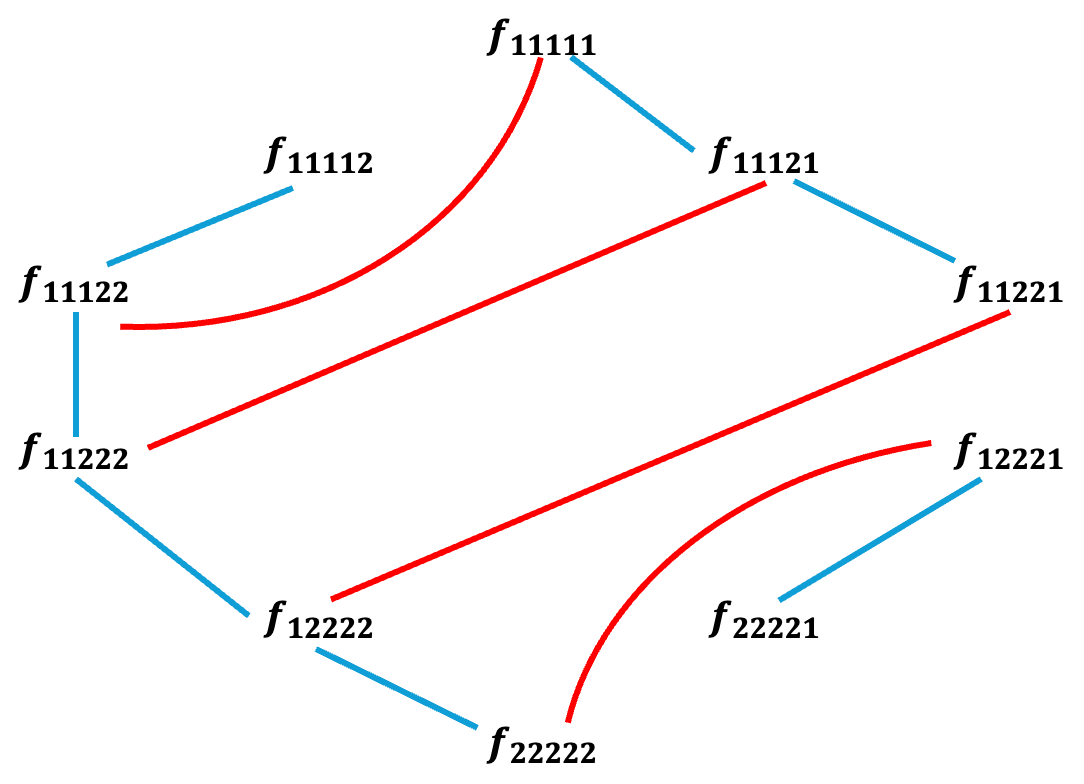}\caption{$m = 5$, $k = 3$, and $\ell = 2$}\label{fig:ANB}
\end{subfigure}\caption{Two components of $f$ are connected with a ``red path" if they are related by the trace-free condition and with a ``blue path" if they are related due to the other condition.}
\end{figure}
\noindent As we see, these graphs are connected, which shows that we can go from any specific component of $f$ to any other component of $f$ by choosing a combination of ``red" and ``blue" paths. One can play the same game for $m$-tensor field in $\mathbb{R}^2$.\\

\section*{Acknowledgements}\label{sec:acknowledge}
 RM was partially supported by SERB SRG grant No. SRG/2022/000947. CT was supported by the PMRF fellowship from the Government of India.
\bibliographystyle{siam}
\bibliography{reference}
\end{document}